\documentclass[12pt,oneside]{amsart}  
\usepackage{amssymb}
\usepackage{amsfonts}
\usepackage{amsthm}
\usepackage{amsthm}
\usepackage{amscd}

\reversemarginpar         
\numberwithin{equation}{section}
\theoremstyle{plain}
\newtheorem{theorem}{Theorem}[section]

\newtheorem{proposition}[theorem]{Proposition}

\newtheorem{lemma}[theorem]{Lemma}
\theoremstyle{definition}

\theoremstyle{remark}
\newtheorem*{remark}{Remark}
\newtheorem*{example}{Example}

\begin{document}
%

\newcommand{\MgNekp}{\mathcal{M}_{g,N+1}^{(k,p)}} 
\newcommand{\M}{\mathcal{M}_{g,N+1}^{(1)}}
\newcommand{\Teich}{\mathcal{T}_{g,N+1}^{(1)}}
\newcommand{\T}{\mathrm{T}}
\newcommand{\corr}{\bf}
\newcommand{\vac}{|0\rangle}
\newcommand{\Ga}{\Gamma}
\newcommand{\new}{\bf}
\newcommand{\define}{\def}
\newcommand{\redefine}{\def}
\newcommand{\Cal}[1]{\mathcal{#1}}
\renewcommand{\frak}[1]{\mathfrak{{#1}}}
\newcommand{\refE}[1]{(\ref{E:#1})}
\newcommand{\refS}[1]{Section~\ref{S:#1}}
\newcommand{\refSS}[1]{Section~\ref{SS:#1}}
\newcommand{\refT}[1]{Theorem~\ref{T:#1}}
\newcommand{\refO}[1]{Observation~\ref{O:#1}}
\newcommand{\refP}[1]{Proposition~\ref{P:#1}}
\newcommand{\refD}[1]{Definition~\ref{D:#1}}
\newcommand{\refC}[1]{Corollary~\ref{C:#1}}
\newcommand{\refL}[1]{Lemma~\ref{L:#1}}
\newcommand{\R}{\ensuremath{\mathbb{R}}}
\newcommand{\C}{\ensuremath{\mathbb{C}}}
\newcommand{\N}{\ensuremath{\mathbb{N}}}
\newcommand{\Q}{\ensuremath{\mathbb{Q}}}
\renewcommand{\P}{\ensuremath{\mathcal{P}}}
\newcommand{\Z}{\ensuremath{\mathbb{Z}}}
\newcommand{\kv}{{k^{\vee}}}
\renewcommand{\l}{\lambda}
\newcommand{\gb}{\overline{\mathfrak{g}}}
\newcommand{\hb}{\overline{\mathfrak{h}}}
\newcommand{\g}{\mathfrak{g}}
\newcommand{\h}{\mathfrak{h}}
\newcommand{\gh}{\widehat{\mathfrak{g}}}
\newcommand{\ghN}{\widehat{\mathfrak{g}_{(N)}}}
\newcommand{\gbN}{\overline{\mathfrak{g}_{(N)}}}
\newcommand{\tr}{\mathrm{tr}}
\newcommand{\sln}{\mathfrak{sl}}
\newcommand{\sn}{\mathfrak{s}}
\newcommand{\so}{\mathfrak{so}}
\newcommand{\spn}{\mathfrak{sp}}
\newcommand{\tsp}{\mathfrak{tsp}(2n)}
\newcommand{\gl}{\mathfrak{gl}}
\newcommand{\slnb}{{\overline{\mathfrak{sl}}}}
\newcommand{\snb}{{\overline{\mathfrak{s}}}}
\newcommand{\sob}{{\overline{\mathfrak{so}}}}
\newcommand{\spnb}{{\overline{\mathfrak{sp}}}}
\newcommand{\glb}{{\overline{\mathfrak{gl}}}}
\newcommand{\Hwft}{\mathcal{H}_{F,\tau}}
\newcommand{\Hwftm}{\mathcal{H}_{F,\tau}^{(m)}}

\newcommand{\car}{{\mathfrak{h}}}    
\newcommand{\bor}{{\mathfrak{b}}}    
\newcommand{\nil}{{\mathfrak{n}}}    
\newcommand{\vp}{{\varphi}}
\newcommand{\bh}{\widehat{\mathfrak{b}}}  
\newcommand{\bb}{\overline{\mathfrak{b}}}  
\newcommand{\Vh}{\widehat{\mathcal V}}
\newcommand{\KZ}{Kniz\-hnik-Zamo\-lod\-chi\-kov}
\newcommand{\TUY}{Tsuchia, Ueno  and Yamada}
\newcommand{\KN} {Kri\-che\-ver-Novi\-kov}
\newcommand{\pN}{\ensuremath{(P_1,P_2,\ldots,P_N)}}
\newcommand{\xN}{\ensuremath{(\xi_1,\xi_2,\ldots,\xi_N)}}
\newcommand{\lN}{\ensuremath{(\lambda_1,\lambda_2,\ldots,\lambda_N)}}
\newcommand{\iN}{\ensuremath{1,\ldots, N}}
\newcommand{\iNf}{\ensuremath{1,\ldots, N,\infty}}

\newcommand{\tb}{\tilde \beta}
\newcommand{\tk}{\tilde \kappa}
\newcommand{\ka}{\kappa}
\renewcommand{\k}{\kappa}

\newcommand{\Pif} {P_{\infty}}
\newcommand{\Pinf} {P_{\infty}}
\newcommand{\PN}{\ensuremath{\{P_1,P_2,\ldots,P_N\}}}
\newcommand{\PNi}{\ensuremath{\{P_1,P_2,\ldots,P_N,P_\infty\}}}
\newcommand{\Fln}[1][n]{F_{#1}^\lambda}
\newcommand{\tang}{\mathrm{T}}
\newcommand{\Kl}[1][\lambda]{\can^{#1}}
\newcommand{\A}{\mathcal{A}}
\newcommand{\U}{\mathcal{U}}
\newcommand{\V}{\mathcal{V}}
\renewcommand{\O}{\mathcal{O}}
\newcommand{\Ae}{\widehat{\mathcal{A}}}
\newcommand{\Ah}{\widehat{\mathcal{A}}}
\newcommand{\La}{\mathcal{L}}
\newcommand{\Le}{\widehat{\mathcal{L}}}
\newcommand{\Lh}{\widehat{\mathcal{L}}}
\newcommand{\eh}{\widehat{e}}
\newcommand{\Da}{\mathcal{D}}
\newcommand{\kndual}[2]{\langle #1,#2\rangle}
\newcommand{\cins}{\frac 1{2\pi\mathrm{i}}\int_{C_S}}
\newcommand{\cinsl}{\frac 1{24\pi\mathrm{i}}\int_{C_S}}
\newcommand{\cinc}[1]{\frac 1{2\pi\mathrm{i}}\int_{#1}}
\newcommand{\cintl}[1]{\frac 1{24\pi\mathrm{i}}\int_{#1 }}
\newcommand{\w}{\omega}
\newcommand{\ord}{\operatorname{ord}}
\newcommand{\res}{\operatorname{res}}
\newcommand{\nord}[1]{:\mkern-5mu{#1}\mkern-5mu:}
\newcommand{\Fn}[1][\lambda]{\mathcal{F}^{#1}}
\newcommand{\Fl}[1][\lambda]{\mathcal{F}^{#1}}
\renewcommand{\Re}{\mathrm{Re}}

\newcommand{\ha}{H^\alpha}

\define\ldot{\hskip 1pt.\hskip 1pt}
\define\ifft{\qquad\text{if and only if}\qquad}
\define\a{\alpha}
\redefine\d{\delta}
\define\w{\omega}
\define\ep{\epsilon}
\redefine\b{\beta} \redefine\t{\tau} \redefine\i{{\,\mathrm{i}}\,}
\define\ga{\gamma}
\define\cint #1{\frac 1{2\pi\i}\int_{C_{#1}}}
\define\cintta{\frac 1{2\pi\i}\int_{C_{\tau}}}
\define\cintt{\frac 1{2\pi\i}\oint_{C}}
\define\cinttp{\frac 1{2\pi\i}\int_{C_{\tau'}}}
\define\cinto{\frac 1{2\pi\i}\int_{C_{0}}}
\define\cinttt{\frac 1{24\pi\i}\int_C}
\define\cintd{\frac 1{(2\pi \i)^2}\iint\limits_{C_{\tau}\,C_{\tau'}}}
\define\cintdr{\frac 1{(2\pi \i)^3}\int_{C_{\tau}}\int_{C_{\tau'}}
\int_{C_{\tau''}}}
\define\im{\operatorname{Im}}
\define\re{\operatorname{Re}}
\define\res{\operatorname{res}}
\redefine\deg{\operatornamewithlimits{deg}}
\define\ord{\operatorname{ord}}
\define\rank{\operatorname{rank}}
\define\fpz{\frac {d }{dz}}
\define\dzl{\,{dz}^\l}
\define\pfz#1{\frac {d#1}{dz}}

\define\K{\Cal K}
\define\U{\Cal U}
\redefine\O{\Cal O}
\define\He{\text{\rm H}^1}
\redefine\H{{\mathrm{H}}}
\define\Ho{\text{\rm H}^0}
\define\A{\Cal A}
\define\Do{\Cal D^{1}}
\define\Dh{\widehat{\mathcal{D}}^{1}}
\redefine\L{\Cal L}
\newcommand{\ND}{\ensuremath{\mathcal{N}^D}}
\redefine\D{\Cal D^{1}}
\define\KN {Kri\-che\-ver-Novi\-kov}
\define\Pif {{P_{\infty}}}
\define\Uif {{U_{\infty}}}
\define\Uifs {{U_{\infty}^*}}
\define\KM {Kac-Moody}
\define\Fln{\Cal F^\lambda_n}
\define\gb{\overline{\mathfrak{ g}}}
\define\G{\overline{\mathfrak{ g}}}
\define\Gb{\overline{\mathfrak{ g}}}
\redefine\g{\mathfrak{ g}}
\define\Gh{\widehat{\mathfrak{ g}}}
\define\gh{\widehat{\mathfrak{ g}}}
\define\Ah{\widehat{\Cal A}}
\define\Lh{\widehat{\Cal L}}
\define\Ugh{\Cal U(\Gh)}
\define\Xh{\hat X}
\define\Tld{...}
\define\iN{i=1,\ldots,N}
\define\iNi{i=1,\ldots,N,\infty}
\define\pN{p=1,\ldots,N}
\define\pNi{p=1,\ldots,N,\infty}
\define\de{\delta}

\define\kndual#1#2{\langle #1,#2\rangle}
\define \nord #1{:\mkern-5mu{#1}\mkern-5mu:}
\define \sinf{{\widehat{\sigma}}_\infty}
\define\Wt{\widetilde{W}}
\define\St{\widetilde{S}}
\newcommand{\SigmaT}{\widetilde{\Sigma}}
\newcommand{\hT}{\widetilde{\frak h}}
\define\Wn{W^{(1)}}
\define\Wtn{\widetilde{W}^{(1)}}
\define\btn{\tilde b^{(1)}}
\define\bt{\tilde b}
\define\bn{b^{(1)}}
%
\define\eps{\varepsilon}    
\define\doint{({\frac 1{2\pi\i}})^2\oint\limits _{C_0}
       \oint\limits _{C_0}}                            
\define\noint{ {\frac 1{2\pi\i}} \oint}   
\define \fh{{\frak h}}     
\define \fg{{\frak g}}     
\define \GKN{{\Cal G}}   
\define \gaff{{\hat\frak g}}   
\define\V{\Cal V}
\define \ms{{\Cal M}_{g,N}} 
\define \mse{{\Cal M}_{g,N+1}} 
\define \tOmega{\Tilde\Omega}
\define \tw{\Tilde\omega}
\define \hw{\hat\omega}
\define \s{\sigma}
\define \car{{\frak h}}    
\define \bor{{\frak b}}    
\define \nil{{\frak n}}    
\define \vp{{\varphi}}
\define\bh{\widehat{\frak b}}  
\define\bb{\overline{\frak b}}  
\define\Vh{\widehat V}
\define\KZ{Knizhnik-Zamolodchikov}
\define\ai{{\alpha(i)}}
\define\ak{{\alpha(k)}}
\define\aj{{\alpha(j)}}
\newcommand{\laxgl}{\overline{\mathfrak{gl}}}
\newcommand{\laxsl}{\overline{\mathfrak{sl}}}
\newcommand{\laxso}{\overline{\mathfrak{so}}}
\newcommand{\laxsp}{\overline{\mathfrak{sp}}}
\newcommand{\laxs}{\overline{\mathfrak{s}}}
\newcommand{\laxg}{\overline{\frak g}}
\newcommand{\bgl}{\laxgl(n)}
\newcommand{\tX}{\widetilde{X}}
\newcommand{\tY}{\widetilde{Y}}
\newcommand{\tZ}{\widetilde{Z}}

\vspace*{-1cm}
%
%
%
\vspace*{2cm}

\large{
\title[Lax equations and Knizhnik-Zamolodchikov connection]
{Lax equations and Knizhnik-Zamolodchikov connection}
\author[O.K. Sheinman]{Oleg K. Sheinman}
\thanks{Supported in part
by the program "Fundamental Problems of Nonlinear Dynamics" of the
Russian Academy of Sciences}


\begin{abstract}
Given a Lax system of equations with the spectral parameter on a
Riemann surface we construct a projective unitary representation
of the Lie algebra of Hamiltonian vector fields by
Knizhnik-Zamolodchikov operators. This provides a prequantization
of the Lax system. The representation operators of Poisson
commuting Hamiltonians of the Lax system projectively commute. If
Hamiltonians depend only on action variables then the
corresponding operators commute
\end{abstract} 
\maketitle


\section{Introduction}\label{S:intro}

In \cite{Klax} I.Krichever proposed a new notion of Lax operator
with a spectral parameter on a Riemann surface. He has given a
general and transparent treatment of Hamiltonian theory of the
corresponding Lax equations. This work has called into being the
notion of Lax operator algebras \cite{KSlax} and consequent
generalization of the Krichever's approach on Lax operators taking
values in the classical Lie algebras over $\C$
\cite{Sh_lopa,Sh0910_Hamilt}. The corresponding class of Lax
integrable systems contains Hitchin systems and their analog for
pointed Riemann surfaces, integrable gyroscopes and similar
examples.

In the present paper we address the following problem: given a Lax
integrable system of the just mentioned type, to construct a
unitary projective representation of the corresponding Lie algebra
of Hamiltonian vector fields. For the Lax equations in question,
we propose a way to represent Hamiltonian vector fields by
covariant derivatives with respect to the Knizhnik-Zamolodchikov
connection. It is conventional that Knizhnik-Zamolodchikov-Bernard
operators provide a quantization of Calogero-Moser and Hitchin
second order Hamiltonians \cite{FeW,D_I}. Unexpectedly, we
observed such relation for all Hamiltonians, and, moreover, for
all observables of the Hamiltonian system given by the Lax
equations in question.

Consider the phase space $\P^D$ of the Lax system in the Krichever
formalism. Every element of $\P^D$ is a meromorphic $\g$-valued
function on a Riemann surface $\Sigma$ satisfying certain
constrains where $\g$ is a classical Lie algebra over $\C$. Assign
every $L\in\P^D$ with its spectral curve $\det(L(z)-\l)~=~0$ which
is a finite branch covering of $\Sigma$. Thus we have a family of
Riemann surfaces over $\P^D$. Following the lines of \cite{WZWN2}
(see also \cite{ShDiss}) we construct an analog of the
Knizhnik-Zamolodchikov connection for this family.  For this end,
we construct a sheaf of admissible representations of the Lax
operator algebras canonically related to the integrable system.
For any tangent vector $X$ to $\P^D$ denote by $\rho(X)$ its image
under the Kodaira-Spencer mapping. In general $\rho(X)\in
H^1(\Sigma_L, T\Sigma_L)$ where $\Sigma_L$ is the spectral curve
at the corresponding point $L\in\P^D$ and $T\Sigma_L$ is its
tangent sheaf. It is shown in \cite{WZWN2} that $\rho(X)$ can be
considered as a Krichever-Novikov vector field, and the arising
ambiguity is compensated by passing to a certain quotient sheaf
called a sheaf of {\it coinvariants}. Hence we may consider
$T(\rho(X))$ where $T$ is the Sugawara representation, and further
on define the high genus Knizhnik-Zamolodchikov connection
\[  \nabla_{X}=\partial_{X}+T(\rho(X)).
\]
By projective flatness of $\nabla$ \cite{WZWN2} we have
\[
[\nabla_{X},\nabla_{Y}]=\nabla_{[X,Y]}+\l(X,Y)\cdot id ,
\]
i.e. $X\to\nabla_X$ is a projective representation of the Lie
algebra of vector fields on $\P^D$.

Now observe that $\P^D$ is a symplectic manifold with respect to
the Krichever-Phong symplectic structure $\w$ \cite{Klax,KrPhong}.
For any $f\in C^\infty(\P^D)$ we can consider the corresponding
Hamiltonian vector field $X_f$. The correspondence $f\to
\nabla_{X_f}$ is a representation of the Poisson algebra of
classical observables of the Lax integrable system. An easy
corollary is as follows. If $\{H_a\}$ is a Poisson commuting
family of Hamiltonians labelled by an index $a$ then the
corresponding Knizhnik-Zamolodchikov operators projectively
commute.

The just introduced representation of the Lie algebra of
Hamiltonian vector fields is unitary. This claim relies on two
facts. First, by Poincar\'e theorem the symplectic form and its
degrees are absolute integral invariants of Hamiltonian phase
flows. Hence the measure on $\P^D$ defined by the volume form
$\w^p/p!$ (where $p=(\dim\P^D)/2$) is invariant with respect to
the Hamiltonian flows. For this reason the $\partial_X$ operator
is skew-symmetric for any real Hamiltonian $X$ in a certain
subspace in $L^2(\P^D,\w^p/p!)$. Second, the Sugawara
representation is unitary as soon as the underlying representation
of the current algebra possesses this property. For the  Virasoro
and loop algebras it is proved in \cite{KaRa}, and for
Krichever-Novikov vector field and commutative current algebras in
\cite{KNFb}.

The subject of the paper certainly must be discussed in the
context of quantization of Lax integrable systems. Our work is
very similar to Hitchin \cite{Hit}. Both are devoted to the
problem of a correspondence between an integrable system and a
connection on a certain moduli space, in a different set-up. In
\cite{Hit} N.J.Hitchin notes two problems of his approach: taking
into account the marked points on Riemann surfaces and unitarity
of the connection. He points out that the Knizhnik-Zamolodchikov
connection could be a solution of the first problem. As it is
shown below, it resolves also the second one.

A large number of works is devoted to quantum integrable systems.
Let us give here a brief outline of the construction due to
B.Feigin and E.Frenkel \cite{FF}, A.Beilinson and V.Drinfeld
\cite{BD}. Denote by $U(\g)_\k$ the universal enveloping algebra
of $\g$ on the level $\k$. There is the only case when $U(\g)_\k$
has an infinite-dimensional center, namely if $\k$ is equal to the
dual Coxeter number for $\g$. By means of the double-coset
construction one can obtain an action of the finite-dimensional
quotient space of the center by differential operators in the
smooth sections of some vector bundle on the moduli space of
holomorphic vector bundles on a Riemann surface. Indeed, $\mathcal
M=G_{left}\setminus G/G_+$ where $\mathcal M$ is the moduli space
of holomorphic vector bundles, $G$ is the loop group corresponding
to $\g$, $G_+$ is its subgroup corresponding to $\gb_+$ and
$G_{left}$ is a certain subgroup of $G_-$. The center of
$U(\g)_\k$ acts on $G$ by casimirs. This action pushes down to
$\mathcal M $. Not every vector bundle on $\mathcal M$ admits a
nontrivial space of differential operators. It exists in a
more-less unique case of $\sqrt{\mathcal K}$ --- the bundle of
half-forms on $\mathcal M$ (where $\mathcal K$ is the canonical
bundle on $\mathcal M$). This construction results in the quantum
Hitchin system in a sense that the symbols of the obtained
operators are equal to the corresponding classical Hitchin
Hamiltonians.

Observe that the authors do not quantize the full algebra of
observables but only its commutative subalgebra (see
\cite[2.2.5]{BD} for example).

Especially detailed information is obtained on quantum
Calogero-Moser systems due to the works of A.P.Veselov,
A.N.Sergeev, G.Felder, M.V.Feigin.

The idea of quantization of Hitchin systems by means
Knizhnik-Zamolodchikov connection was also addressed, or at least
mentioned, many times in the theoretical physics literature
(D.Ivanov \cite{D_I}, G.Felder and  Ch.Wieczerkowski \cite{FeW},
M.A.Olshanetsky and A.M.Levin) but only the second order
Hamiltonians have been involved.

To conclude with, let us note that the results of the present
paper provide a prequantization of the Lax integrable systems with
the spectral parameter on a Riemann surface. We prequantize the
whole algebra of observables rather than  any commutative
subalgebra of it.

The paper is organized as follows. In \refS{phase} we give a
description of Lax integrable systems with the spectral parameter
on a Riemann surface. This section is a survey of the results of
\cite{Klax,{KSlax},{Sh_lopa},{Sh0910_Hamilt}}. We introduce the
phase parameters and explain their relation to the Tjurin
parametrization of holomorphic vector bundles on Riemann surfaces.
Then we define the notion of a Lax operator with the spectral
parameter on a Riemann surface, and values in a classical complex
Lie algebra. We introduce the corresponding Lax equations,
construct their hierarchy of commuting flows and present their
Hamiltonian theory including the  Krichever-Phong symplectic
structure.

In \refS{CFT}, we introduce a conformal field theory related to an
integrable system of the above described type. By conformal field
theory we mean a family of Riemann surfaces, a finite rank bundle
(of coinvariants) on this family, and a flat connection (the
Knizhnik-Zamolodchikov connection) on this bundle.

As a family of Riemann surfaces we take the family of spectral
curves corresponding to the Lax integrable system in question. We
recall from \cite{WZWN2,{ShDiss}} the above mentioned version of
the Kodaira-Spencer mapping. Further on we introduce a certain
commutative current Krichever-Novikov algebra which is required by
the Sugawara construction. For a Lax operator algebra this
commutative algebra plays a role similar to Cartan subalgebra in
the Kac-Moody and Cartan-Weil theories. At last, we recall from
\cite{Sh_ferm} the construction of the fermionic representation of
Krichever-Novikov current algebras and carry out the Sugawara
construction.

In \refS{repr} we we formulate and prove main results of the
paper. We construct the Knizhnik-Zamolodchikov connection on the
family of spectral curves and prove that the
Knizhnik-Zamolodchikov operators give a projective unitary
representation of the Lie algebra of Hamiltonian vector fields. As
a corollary we obtain that the operators corresponding to the
family of commuting Hamiltonians commute up to scalar operators.

The author is grateful to I.M.Krichever and to M.Schlichenmaier
for many fruitful discussions. The majority of authors results the
present work relies on are obtained in collaboration with them. I
am also grateful to D.Talalaev. Our discussions were helpful in
order to realize a role of the spectral curve in quantization. I
am thankful to A.P.Veselov and M.A.Olshanetsky for useful
discussions.

\section{Phase space and Hamiltonians of a Lax integrable system}\label{S:phase}

In the present section following the lines of
\cite{Klax,Sh0910_Hamilt} we consider a certain class of
integrable systems given by $\g$-valued (in particular
matrix-valued) Lax operators of zero order with a spectral
parameter on a Riemann surface. The examples include
Calogero-Moser systems, Hitchin systems and their generalizations,
gyroscopes etc.

\subsection{Geometric data} Every integrable system in question
is given by the following geomerical data: a Riemann surface
$\Sigma$ with a given complex structure, a classical Lie algebra
$\g$ over $\C$, fixed points $P_1,\ldots,\linebreak P_N\in\Sigma$
($N\in\Z_+$), a positive divisor $D=\sum\limits_{i=1}^{N}m_iP_i$,
points $\ga_1,\ldots, \ga_{K}\in\Sigma$ ($K\in\Z_+$), vectors
$\a_1,\ldots, \a_{K}\in\C^n$ associated with $\ga$'s and given up
to the right action of a classical group $G$ corresponding to
$\g$. The last two items ($\ga$'s and $\a$'s) are joined under the
name {\it Tyurin data}, because of the following
\begin{theorem}[A.N.Tyurin]
Let $g=genus\ \Sigma$, $n\in\Z_+$. Then there is a 1-to-1
correspondence between the following data:
\newline
 1) \ points $\ga_1,\ldots,\ga_{ng}$ of \ $\Sigma$ ;
 \newline
 2) \ $\a_1,\ldots,\a_{ng}\in\C P^{n-1}$
 \newline
and the equivalence classes of the equipped semi-stable
holomorphic rank $n$ vector bundles on $\Sigma$
\end{theorem}
\noindent where equipment means fixing $n$ holomorphic sections
linear independent except at $ng$ points.

\subsection{Lax operators on Riemann surfaces}

Let $\{\a\}=\{ \a_1,\ldots, \a_{K} \}$,   $\{\ga\}=\{\ga_1,\ldots,
\ga_{K}\}$, $\{\k\}=\{ \k_i\in\C |i=1,\ldots,K \}$. Below, we will
avoid the indices using $\a$ instead $\a_i$ etc. Consider a set
$\{\b\}=\{ \b_1,\ldots, \b_{K} \}$ dual to $\{\a\}$ with respect
to the symplectic structure to appear below.

Consider a function $L(P,\{\a\},\{\b\},\{\ga\},\{\k\})$
($P\in\Sigma$) obeying certain requirements which will follow
immediately. In the local coordinate $z$ on $\Sigma$ we refer to
this function as to $L(z)$ omitting the indication to other
arguments. We require that $L$ is meromorphic as a function of
$P$, has arbitrary poles at $P_i$'s, simple or double poles at
$\ga$'s (depending on $\g$), is holomorphic elsewhere, and at
every $\ga$ is of the form
\[
 L(z)=\frac{L_{-2}}{(z-z_\ga)^2}+\frac{L_{-1}}{(z-z_\ga)}+L_{0}
 +L_1(z-z_\ga)+O((z-z_\ga)^2)
\]
where $z$ is a local coordinate at $\ga$,  $z_\ga=z(\ga)$ and the
following relations hold:
\begin{equation}\label{E:rel1}
L_{-2}=\nu\a\a^t\s , \ \  L_{-1}=(\a\b^t+\eps\b\a^t)\s, \ \
\b^t\s\a=0, \ \ L_0\a=\k\a
\end{equation}
where $\a,\b\in \C^n$ ($\a$ is associated with $\ga$, $\b$ is
arbitrary), $\nu\in\C$, $\s$ is a $n\times n$ matrix. $L$ is
called a {\it  Lax operator with a spectral parameter on the
Riemann surface $\Sigma$}. The $\nu$, $\eps$, $\s$ in \refE{rel1}
depend on $\g$ as follows:
\begin{equation}\label{E:llaeps}
  \begin{aligned}
     \nu\equiv 0,\ &\eps=0,\ \ \  \s=id\ \ \text{for}\ \g=\gl(n),\sln(n),\\
     \nu\equiv 0,\ &\eps=-1,\ \s=id\ \ \text{for}\ \g=\so(n), \\
                  &\eps=1\ \ \ \ \ \ \ \ \ \ \ \ \ \ \
                  \text{for}\ \g=\spn(2n),
  \end{aligned}
\end{equation}
and $\s$ is a matrix of the symplectic form for $\g=\spn(2n)$.

In addition we assume that
\begin{equation}\label{E:add1}
\a^t\a=0\ \ \text{for}\ \ \g=\so(n)
\end{equation}
and
\begin{equation}\label{E:add2} \a^t\s L_1\a=0\ \ \text{for}\ \
\g=\spn(2n).
\end{equation}


\subsection{Lax operator algebras}

\begin{theorem}[Lie algebra structure, \cite{KSlax}]  For
fixed Tyurin data the space of Lax operators is closed with
respect to the point-wise commutator $[L,L'](P)=[L(P),L'(P)]$
$(P\in\Sigma)$ (in the case $\g=\gl(n)$ also with respect to the
point-wise multiplication).
\end{theorem}
It is called {\it Lax operator algebra} and denoted by $\gb$.
\begin{theorem}[almost graded structure, \cite{KSlax}]\label{T:almgrad} There exist
such finite-dimensional subspaces $\g_m\subset \g$ that
\[
(1)\ \gb=\bigoplus\limits_{m=-\infty}^\infty \g_m ; \ \ \ (2)
\dim\,\g_m=\dim\,\g ; \ \ \  (3)\
[\g_k,\g_l]\subseteq\bigoplus\limits_{m=k+l}^{k+l+ g}\g_m .
\]
\end{theorem}
\begin{theorem} If $\g$ is
simple then $\gb$ has only one almost graded central extension, up
to equivalence \cite{SSlax}. It is given by a cocycle
$\gamma(L,L')=-\res_{P_\infty}\tr(LdL'-[L,L']\theta)$ where
$\theta$ is a certain 1-form \cite{KSlax}.
\end{theorem}

\subsection{M-operators}

$M=M(z,\{\a\},\{\b\},\{\ga\},\{\k\})$ is defined by the same
constrains as $L$, excluding $\b^t\s\a=0$ and $L_0\a=\k\a$, namely
\[
 M=\frac{M_{-2}}{(z-z_\ga)^2}+\frac{M_{-1}}{z-z_\ga}+M_{0}+M_1(z-z_\ga)
 +O((z-z_\ga)^2)
\]
where
\begin{equation}\label{E:M-con}
 M_{-2}=\l\a\a^t\s , \ \  M_{-1}=(\a\mu^t+\eps\mu\a^t)\s
\end{equation}
$M$ also takes values in $\g$, $\l\in\C$, $\mu\in\C^n$.

\subsection{Lax equations}

For variative Tyurin data, the collection of equations on
$\{\a\}$, $ \{\b\}$, $\{\ga\}$, $\{\k\}$ and main parts of $L$ at
$\{P_i|i=1,\ldots,N\}$ equivalent to the relation
\begin{equation}\label{E:LaxEq}
                \dot L=[L,M]
\end{equation}
is called a {\it Lax equation}.

{\it Motion equations of Tyurin data} assigned to a point $\ga$:
\[
   \dot z_\ga=-\mu^t\s\a, \ \ \dot\a=-M_0\a+k_a.
\]
Besides, there are motion equations of main parts of the function
$L$ at $P_i$'s.

Let $D=\sum m_iP_i$ ($i=1,\ldots, N,\infty$) be a divisor such
that $\text{supp}\, D\cap\{\ga\}=\varnothing$, $\L^D:=\{L|(L)+D\ge
0\ \text{outside}\ \ga\text{'s}\}$. Stress again that the elements
of $\L^D$ have a two-fold interpretation: as Lax operators and as
sets of Tyurin data and main parts of $L$-operators.

Under a certain (effective) condition \cite{KSlax,SSlax} the Lax
equation defines a flow on $\L^D$.


\subsection{Examples}

1) $g=0$, $\a=0$  (i.e. $\Sigma=\C P^1$, the bundle is trivial),
$P_1=0$, $P_2=\infty$. Then $\gb=\g\otimes\C[z,z^{-1}]$ is a loop
algebra, \refE{LaxEq} is a conventional Lax equation with
rational spectral parameter:
\[L_t=[L,M], \quad L,M\in \g\otimes\C[\l^{-1},\l), \quad \l\in{\mathcal
D}^1.
\]
The Lax equations of this type are considered by I.Gelfand,
L.Dikii, I.Dorfman, A.Reyman, M.Semenov-Tian-Shanskii, V.Drinfeld,
V.Soko\-lov, V.Kac, P. van Moerbeke. Many known integrable cases
of motion and hydrodynamics of a solid body belong to this class.

2) Elliptic curves: Calogero-Moser systems \cite{Sh0910_Hamilt}.

\vskip5pt 3) Arbitrary genus: Hitchin systems

\subsection{Hierarchy of commuting flows}

\begin{theorem}[\cite{Klax,Sh_lopa,Sh0910_Hamilt}]\label{T:hierarch}
Given a generic $L$ and effective divisor $D=\sum m_iP_i$
($i=1,\ldots, N,\infty$, there is a family of $M$-operators
$M_a=M_a(L)$ ($a=(P_i,n,m), n>0,\ m>-m_i$) uniquely defined up to
normalization, such that outside the $\ga$-points $M_a$ has pole
at the point $P_i$ only, and in the neighborhood of $P_i$
\[
  M_a(w_i)=w_i^{-m}L^n(w_i)+O(1),
\]

The equations
 \begin{equation}\label{E:hi}
   \partial_aL=[L,M_a],\ \partial_a=\partial/\partial t_a
 \end{equation}
 define a family of commuting flows on an open set of $\L^D$.
\end{theorem}

\subsection{Krichever-Phong symplectic structure}

We define an external 2-form on $\L^D$. For $L\in\L^D$ let $\Psi$
be a matrix-valued function formed by the eigenvectors of $L$:
$\Psi L=K\Psi$ ($K$ --- diagonal).
\[\Omega:=\tr(\d\Psi\wedge\d L\cdot\Psi^{-1}-\d K\wedge\d\Psi\cdot\Psi^{-1})
=2\d\tr(\d\Psi\cdot\Psi^{-1}K)
\]
where $\d\Psi$ is the differential of $\Psi$ in $\a,\b,\ldots$.

Let $dz$ be a holomorphic 1-form on $\Sigma$ and
\[\w:=-\frac{1}{2}\left(\sum\res_{\ga_s}\Omega dz+\sum\res_{P_i}\Omega dz\right) .\]

\begin{theorem}[\cite{Klax,Sh0910_Hamilt}] $\w$ is  a symplectic form
on a certain closed invariant submanifold $\P^D\subset\L^D$.
\end{theorem}

\subsection{Hamiltonians}

\begin{theorem}[\cite{Klax,Sh0910_Hamilt}] The equations of the above
commutative family are Hamiltonian with respect to the {\it
Krichever-Phong symplectic structure} on $\P^D$, with the
Hamiltonians given by
\[
  H_a=-\frac{1}{n+1}\res_{P_i}tr(w_i^{-m}L^{n+1})dw_i
\]
\end{theorem}

\begin{example} Let $\g=\gl(n)$, $D$ be a divisor of a holomorphic
1-form. Then $\L^D\simeq T^*({\mathcal M}_0)$ where ${\mathcal
M}_0$ is an open subset of the moduli space of holomorphic vector
bundles on $\Sigma$, $H_a$ are Hitchin Hamiltonians.
\end{example}


\section{Conformal field theory related to a Lax integrable system}
\label{S:CFT}

By conformal field theory we mean a family of Riemann surfaces, a
finite rank bundle (of {\it conformal blocks}) on this family, and
a flat connection ({\it Knizhnik-Zamolodchikov connection}) on
this bundle.

We consider the universal bundle of spectral curves over the phase
space $\P^D$ of the integrable system as a family of Riemann
surfaces. For this family we introduce the Kodaira-Spencer mapping
sending tangent vectors on the base ($\P^D$ in our case) to the
vector fields on the correspond\-ing Riemann surfaces of the
family. We use the Sugawara construction to obtain the analog of
Knizhnik-Zamolodchikov connection on this family. The
Knizhnik-Zamolodchikov operators give a projective representation
of the Lie algebra of Hamiltonian vector fields. We prove that
operators corresponding to the family of commuting Hamiltonians
commute up to scalar operators.

In this section we realize the first part of the programme. The
second part starting from construction of the
Knizhnikov-Zamolodchikov connec\-tion is delayed to the next
section.

\subsection{Spectral curves}\label{SS:spec}
For every $L\in\P^D$ (i.e. all arguments of $L$ are fixed except
for $z$) let $\Sigma_L$ be a curve given by the equation
$\det(L(z)-\l)=0$. $\Sigma_L$ is called a {\it spectral curve} of
$L$. It is a $n$-fold branch covering of $\Sigma$.

Thus a family of curves parameterized by points of $\P^D$ is
obtained. We apply a conformal field theory technique (see
\cite{WZWN1,WZWN2,ShN65,ShDiss} and references there) to construct
a bundle of conformal blocks with a projective flat connection on
$\P^D$. We will represent a hamiltonian vector field on $\P^D$ by
the operator of covariant derivative along that.

\subsection{Krichever-Novikov vector fields}\label{SS:KNVir}

Consider an arbitrary Riemann surface with marked points and the
Lie algebra of meromorphic vector fields on it holomorphic outside
the marked points. It is called {\it Krichever-Novikov vector
field algebra} and denoted by $\V$. Below we consider
Krichever-Novikov vector field algebra on $\Sigma_L$ with the
preimages of $P_1,\ldots,P_N$ as marked points. With every $P_i$
and an arbitrary $n\in\N$ we associate an element $e_{i,n}$ such
that the collection of those form a base in $\V$ as of a vector
space, and
\begin{equation*}
[e_{n,p},e_{m,r}]=\d_p^r\,(m-n)\,e_{n+m,p}+
\sum_{h=n+m+1}^{n+m+l}\sum_{s=1}^N\gamma_{(n,p),(m,r)}^{(h,s)}
e_{h,s}
\end{equation*}
with some $\gamma_{(n,p),(m,r)}^{(h,s)}\in\C$, $l\in\N$ (we refer
to \cite{WZWN1,WZWN2,ShN65,ShDiss} for details). Hence the
subspaces $\V_n=\bigoplus\limits_{i=1}^N\C e_{i,n}$ give an almost
graded structure on $\V$ in the same sense as in the
\refT{almgrad}.

\subsection{Kodaira-Spencer cocycle}\label{SS:Kodaira}
Denote the spectral curve over $L$ by $\Sigma_L$ and the Lie
algebra of Krichever-Novikov vector fields on $\Sigma_L$ by
$\V_L$. Our next goal is to define a map $\rho : \T_L\P^D\to\V_L$.
Fix a certain point, say $P_\infty\in\Sigma_L$. We may think of
$P_\infty$ as of analitically depending on $L$. Choose a local
family of transition functions $d_L$ giving the complex structure
on $\Sigma_L$ and analitically depending on $L$. Let us take $X\in
\T_L\P^D$ and a curve $L_X(t)$ in $\P^D$ with the initial point
$L$ and the tangent vector $X$ at $L$. By definition
\begin{equation}\label{E:KodSp2}
   \rho(X)=d_L^{-1}\cdot\partial_Xd_L.
\end{equation}
We consider $\rho(X)$ as a local vector field on the Riemann
surface $\Sigma_L$. It can be explained in two ways. First, every
transition function $d_L$ can be considered as a local analitic
diffeomorphism with annulus domain of definition. Such
diffeomorphisms form a group provided two diffeomorph\-isms
coinsiding in a domain are identified. Hence $\rho(X)$ is a
tangent vector at the unit of the group, i.e. a local vector
field, in a standard way. The relation \refE{KodSp2} is an
invariant definition of that vector field. Now let us give a
definition using local coordinates. Let $\tau$ be a set of local
coordinates on $\L^D$. A family of local diffeomorphisms $d_\tau$
can be considered as an analitic function $d(z,\tau)$ where $z$ is
a local coordinate in the neighborhood of $P_\infty$ on
$\Sigma_L$. For $X=\sum_i X_i
\frac{\partial}{\partial\tau_i}\in\T_\tau\P^D$
\begin{equation}\label{E:ueno}
    (\partial_Xd_\tau)(z)=\sum_i X_i\frac{\partial
    d(z,\tau)}{\partial\tau_i}
\end{equation}
is a local function. It gives a tangent vector to the group of
local diffeomor\-phisms at the point $d_\tau$. Now set
\begin{equation}\label{E:ueno1}
    \rho_z(X)= d_\tau^{-1}(\partial_Xd_\tau(z))
    \ \text{and}\  \rho(X)=\rho_z(X)\frac{\partial}{\partial
    z}
\end{equation}
where $\rho_z(X)$ is a local function, and $\rho(X)$ is a local
vector field expressed in the coordinate $z$.

Summarizing the results of \cite[Sect. 5.1]{ShDiss} we obtain
\begin{proposition}\label{P:kernel}
There exist $e\in\V$ such that locally (in the neighborhood of
$P_\infty$) $\rho(X)=e$. The vector field $e$ is defined up to
adding elements of $\V^{(1)}\oplus\V^{\rm reg}$ where $\V^{(1)}$
is the direct sum of homogeneous subspaces of degree $\ge 0$ in
$\V$, and $\V^{\rm reg}\subset\V$ is the subspace of vector fields
having zero at $P_\infty$. Both $\V^{(1)}$ and $\V^{\rm reg}$ are
Lie subalgebras.
\end{proposition}
Below, we always regard to $\rho(X)$ as to an element of $\V^{\rm
reg}\backslash\V/\V^{(1)}$.

A local vector field in the annulus centered at $P_\infty$ gives a
class of Chech $1$-cohomologies of the Riemann surface $\Sigma$
with coefficients in the tangent sheaf. The cohomology class
represented by the vector field $\rho(X)$ is called the {\it
Kodaira-Spencer class} of $X$. It is responsible for the
deformation of moduli of the pointed surface along $X$. In this
form the Kodaira-Spencer class was used for example in \cite[Lemma
1.3.8]{rUcft}, and also in \cite{ShDiss}.

\subsection{From Lax operator algebra to commutative
Krichever-Novikov algebra}\label{SS:aff_KN}

In this section, we canonically associate a commutative
Krichever-Novikov algebra to a generic element $L\in\gb$. We need
that for the Sugawara construction below. Indeed, the Sugawara
construction \cite{KaRa,{KNFa},{WZWN2}} requires that the current
algebra splitted to the tensor product of a functional algebra and
a finite dimensional Lie algebra. Krichever-Novikov algebras are
of this type, and Lax operator algebras are not.

Let $\Psi$ be the matrix formed by the canonically normalized left
eigenvectors of $L$. In the case $\g=\gl(n)$ we consider a vector
$\psi$ to be canically normalized if $\sum\psi_i=1$
\cite{Klax,Sh0910_Hamilt}. In the other cases we require that
$\Psi\in G$ point-wise, i.e. $\Psi^t=-\eps\s\Psi^{-1}\s^{-1}$
where $G=SO(n), Sp(2n)$ depending on $\g$, and $\eps$ satisfies to
$\s^t=-\varepsilon\s$, i.e. $\varepsilon=-1$ if $\s$ is symmetric
(the case $\g=\so(n)$), $\varepsilon=1$ if $\s$ is skew-symmetric
(the case $\g=\spn(2n)$). As usual, $\eps=0$ for $\g=\gl(n)$.
$\Psi$ is defined modulo normalization and permutations of its
rows (such normalization descends to the left multiplication
$\Psi$ by a diagonal matrix). We also consider the diagonal matrix
$K$ defined by
\begin{equation}\label{E:eigenL}
   \Psi L=K\Psi,
\end{equation}
i.e. formed by the eigenvalues of $L$. Similar to
\cite{Klax,Sh0910_Hamilt}
\begin{equation}\label{E:inverse}
\begin{aligned}
\Psi(z)&=\frac{\varepsilon{\tilde\b}\a^t\s}{z-z_\ga}
+\Psi_0+\Psi_1(z-z_\ga)+\ldots\ , \\
\Psi^{-1}(z)&=\frac{\a{\tilde\b}^t\s}{z-z_\ga}+\tilde\Psi_0+
\tilde\Psi_1(z-z_\ga)+\ldots
\end{aligned}
\end{equation}
in the neighborhood of a $\ga$. The residue of $\Psi$ which is
absent in \cite{Klax,Sh0910_Hamilt} appears here for $\g=\so(n)$
and $\g=\spn(2n)$ due to the requirement
$\Psi^t=-\s\Psi^{-1}\s^{-1}$. As it is shown in
\cite{Sh0910_Hamilt} the following relations hold which are
essentially equivalent to holomorphy of $\Psi L$ and the relation
$\Psi\Psi^{-1}=id$:
\begin{equation}\label{E:gpos}
 \Psi_0\a=0, \ \varepsilon\a^t\s{\tilde\Psi}_0=0.
\end{equation}
Observe that if $\varepsilon=0$ (i.e. $\g=\gl(n)$) then
$\nu=\l=0$. Hence
$\nu\a^t\s{\tilde\Psi}_0=\l\a^t\s{\tilde\Psi}_0=0$ as well. For
the same reason $\varepsilon\a^t\s\a=\nu\a^t\s\a=\l\a^t\s\a=0$.
The following Lemma is important only as a motivation for the next
\refL{comp}, no proof below relies on it.
\begin{lemma}[\cite{Sh0910_Hamilt}, Lemma 7.2]\label{L:holK}
The matrix-valued function $K$ is holo\-morphic at all
$\ga$-points provided \refE{inverse},\refE{gpos} hold there.
\end{lemma}
Hence $K$ is a meromorphic diagonal matrix-valued function on
$\Sigma$ holomorphic outside $P_i$'s. If we denote the algebra of
scalar-valued functions possessing this property by $\A$ then
$K\in\h\otimes\A$ where $\h\subset\g$ is a diagonal (Cartan)
subalgebra. $\A$ is called {\it Krichever-Novikov function
algebra} on $\Sigma$, and $\h\otimes\A$ the corresponding {\it
Krichever-Novikov current algebra}. We obtain only commutative
current algebras here.  Let be $\hb=\h\otimes\A$.

In what follows we need a slightly different set-up. We consider
the algebra $\A_L$ similar to $\A$ but defined on $\Sigma_L$, and
having pre-images of the points $P_i$ as the collection of poles.
Let us take an arbitrary element of $\A_L$ and push it down to
 $\Sigma$ as a diagonal matrix $h$. Every sheet is assigned with
 a certain row of $h$. At the branching points we may
 obtain a coincidence of eigenvalues of $h$. The order of diagonal elements of the
 matrix $h$ will depend on the
 order of the sheets. This ambiguity is the same as for $\Psi$.
 The permutation of rows corresponding to an element $w$ of the Weil
 group descends to $\Psi\to w\Psi$ (which is easy verified for $w$
 to be a transposition), and $h$ transforms as $h\to whw^{-1}$.
 Thus $L=\Psi^{-1}h\Psi\to L$. Thus the mapping $\A_L\to\hb$ is
 just the direct image of functions from $\A_L$.

\subsection{The representation of $\hb$}\label{SS:repr}
We need a representation of $\hb$ in order to construct the bundle
of conformal blocks and Knizhnik-Zamolod\-chikov connection in the
next section. Recall that $\gb$ stays for the Lax operator algebra
in question.
\begin{lemma}\label{L:comp}
For any $h\in\hb$ we have $\Psi^{-1} h\Psi\in\gb$.
\end{lemma}

\begin{proof} Let us take $h$ in the form
\[ h=h_0+h_1z+\ldots
\]
in a neighborhood of a $\ga$ where we set $z_\ga=0$ for
simplicity. By \refE{inverse} we obtain
\begin{equation}\label{E:L_in_g}
\begin{aligned}
 & L =\Psi^{-1}
 h\Psi =\frac{\varepsilon\a\tilde\b^t\s h_0\tilde\b\a^t}{z^2}+
 \frac{\a\tilde\b^t\s h_0\Psi_0+\varepsilon\tilde\Psi_0h_0\tilde\b\a^t\s}{z}
 \\& +(\a\tilde\b^t\s h_0\Psi_1
 +\a\tilde\b^t\s h_1\Psi_0+\varepsilon\a\tilde\b^t{\s}h_2\tilde\b^t\a^t\s
 +\tilde\Psi_0 h_0\Psi_0 \\&
 +\varepsilon\tilde\Psi_0 h_1\tilde\b^t\a^t\s+
 \varepsilon\tilde\Psi_1 h_0\tilde\b^t\a^t\s)
 \\& +\ldots\ .
\end{aligned}
\end{equation}
If $\g=\so(n)$ then $\s^t=\s$. We have $(\s h_0)^t=h_0^t\s^t=-\s
h_0\s^{-1}\cdot\s=-\s h_0$, i.e. $\s h_0$ is skew-symmetric. Hence
$\tilde\b^t(\s h_0)\tilde\b=0$ and the term with $z^{-2}$ in
\refE{L_in_g} vanishes. Observe that for $\g=\spn(2n)$ it does not
vanish because $\s h_0$ is symmetric.

Let us consider the term with $z^{-1}$. We must represent the
nominator in the form
$(\a\widehat\b^t+\varepsilon\widehat\b\a^t)\s$, i.e. prove that
$\a\tilde\b^t\s h_0\Psi_0\s^{-1}
=(\tilde\Psi_0h_0\tilde\b\a^t)^t$. By $h_0^t=-\s h_0\s^{-1}$ it
would be sufficient that $\tilde\Psi_0^t=-\s \Psi_0\s^{-1}$. The
latter follows from the requirement $(\Psi^{-1})^t=-\s
\Psi\s^{-1}$, i.e. $\Psi\in G$ where $G$ is a classical group
corresponding to $\g$, and which is included into the definition
of $\Psi$ above.

Next, let us check the eigenvalue condition on $L_0$. To this end,
let us operate on $\a$ by the six-term expression in bracket in
\refE{L_in_g}. Three summands containing the combination
$\varepsilon\a^t\s$ will vanish by $\varepsilon\a^t\s\a=0$. Two
more summands having $\Psi_0$ on the right hand side will vanish
by $\Psi_0\a=0$. The remaining summand will give $\a(\tilde\b^t\s
h_0\Psi_1\a)$. Observe that the expression in brackets is a
scalar.

To complete the proof of the theorem we must check that $\Psi^{-1}
h\Psi$ is holomorphic except at $P_i$'s and $\ga_s$'s.  $\Psi$
possesses certain singular points due to the norm\-ali\-zation
($\sum\psi_i=1$ where $\psi$ is a row of $\Psi$). Let us prove
that $\Psi^{-1} h\Psi$ is holomorphic at those points for any
$h\in\hb$. By \refE{eigenL} $\Psi$ is defined up to left
multiplication by a diagonal matrix $d$ (corresponding to
normalization of the rows of $\Psi$). But since $h$ and $d$ are
both diagonal, $(d\Psi)^{-1} hd\Psi=\Psi^{-1} h\Psi$, hence
$\Psi^{-1} h\Psi$ does not depend on the normali\-zation and can
be defined without any normalization. But then $\Psi$ has no
singularities which are due to normalization.
\end{proof}

By \refL{comp} any representation of $\gb$ induces the
corresponding representation of $\hb$. Since $\Psi$ is meromorphic
at $P_i$'s, the just constructed mapping $\hb\to\gb$ preserves
degree. Hence an almost graded $\g$-module induces the
almost-graded $\hb$-module. Let us remind that $\Psi$ is defined
up to left multiplication by a diagonal matrix. For this reason
$\hb$ needs to be diagonal, otherwise the mapping $\hb\to\gb$ is
not well-defined.

Consider the following natural representation of $\gb$. Let
$\mathcal F$ be the space of meromorphic vector-valued functions
$\psi$ holomorphic except at $P_1,\ldots,P_N$ and $\ga$'s, such
that
\[
 \psi(z)=\nu\frac{\a}{z} +\psi_0+\ldots
\]
at any point $\ga$. $\mathcal F$ is an almost graded $\gb$-module
with respect to the Krichever-Novikov base introduced in
\cite{Sh_ferm}. Consider the semi-infinite degree of this module
which is also constructed in \cite{Sh_ferm}. Denote it by
${\mathcal F}^{\infty/2}$. The induced $\hb$-module is what we
will consider below. This is an {\it admissible} module in the
sense that every its element annihilates having been multiply
operated by an element of $\hb$ of a positive degree. Moreover, it
is generated by a vacuum vector. By the above constructed mapping
$\A_L\to\hb$ we also consider ${\mathcal F}^{\infty/2}$ as an
$\A_L$-module.


\subsection{Sugawara representation}\label{SS:sugawara}

In this paper we need a "commutative"\ version of Sugawara
construction developed in \cite{KNFb} to be applied to $\hb$. See
\cite{WZWN2,ShDiss,ShN65} for more details.

Consider an admissible $\hb$-module $V$. We are mainly interested
in the case $V={\mathcal F}^{\infty/2}$ here.

For any admissible representation of an affine Krichever-Novikov
algebra (say $\hb$) there exist a projective representation T of
the Krichever-Novikov vector field algebra canonically defined by
the relation
\[  [T(e),h(A)]=-c\cdot h(eA)
\]
where $h\in\hT$, $A\in\A_L$, $e\in\V_L$, $h(A)$ denotes the
representation operator of $h\otimes A\in\hT$, $eA$ denotes the
natural action of a vector field on a function, $c$ is a level of
the $\hb$-module (in the non-commutative case we would have $c+\k$
instead $c$ where $\k$ is the dual Coxeter number).

The representation $T$ has the following effective definition. Let
$\{A_j\}$, $\{\w^k\}$, $\{e_m\}$ be the Krichever-Novikov bases in
$\A_L$, the space of Krichever-Novikov 1-formes and
Kricher-Novikov vector fields on $\Sigma_L$ (the first two are
dual), $\{h_i\}$ and $\{h^i\}$ is a pair of dual bases in $\h$ and
$\h^*$. Then $T(e_m)=\sum_{i,j,k} c^{jk}_m\nord{h_i(A_j)h^i(A_k)}$
where $c^{jk}_m=\res_{P_\infty}\w^j\w^ke_m$, $\nord{}$ is a normal
ordering.
\begin{remark} There would be two problems if we wanted to repeat the
Sugawara construction for Lax operator algebras. The first is
related to generalization of the relation $[T(e),u(L)]=u(e.L)$.
The right hand side must contain a connection $\nabla=d+\w$ and
have the form $[T(e),u(L)]=u(\nabla_e L)$ but it is not evident
how $\nabla$ enters the left hand side. Another problem is related
to the definition of $c^{jk}_m$, namely how make the corresponding
1-form $\w^j\w^ke_m$ to be regular at the $\ga$-points. Another
face of the problem is that we need a current algebra to be
splitted to tensor product of a finite-dimensional and a function
algebra to carry out the Sugawara construction.

By introducing $\hb$ and $\A_L$ we manage with the known version
of Sugawara construction for Krichever-Novikov algebras, and even
for commutative Krichever-Novikov algebras.

The role of $L$ is to define an admissible representation of
$\hb$, $\A_L$. Observe that the Lax operator $L$ is fixed as a
function on $\L^D$ and of the spectral parameter, hence $\Psi$ is
uniquely defined up to the above discussed equivalence, and the
representation is well-defined.
\end{remark}

\section{Representation of the algebra of Hamiltonian vector fields}\label{S:repr}
In this section we construct the Knizhnik-Zamolodchikov connection
on the family of spectral curves. The Knizhnik-Zamolodchikov
operators give a projective representation of the Lie algebra of
Hamiltonian vector fields. We prove that operators corresponding
to the family of commuting Hamiltonians commute up to scalar
operators. We prove unitarity of that representation.

\subsection{Conformal blocks and Knizhnik-Zamolodchikov connection}\label{S:cb}
Let us consider the sheaf of $\A_L$-modules ${\mathcal
F}^{\infty/2}$ on $\P^D$. Let $\h^{reg}\subset\hb$ be a subalgebra
consisting of the functions regular at $P_\infty$. The sheaf of
quotients ${\mathcal F}^{\infty/2}/\h^{reg}$ on $\P^D$ is called
the sheaf of {\it covariants} (over a different base it was
defined in \cite{WZWN2} in this way).

Let $X$ be a vector field on $\P^D$. By definition
\begin{equation}\label{E:KZ}
\nabla_X=\partial_X+T(\rho(X))
\end{equation}
where $\rho$ is the Kodaira-Spencer mapping, $T$ is the Sugawara
representation in ${\mathcal F}^{\infty/2}/\h^{reg}$.
\begin{theorem}[\cite{WZWN2,ShDiss}]\label{T:pr_flat} $\nabla$ is
a projective flat connection on the sheaf of coinvariants:
\[
 [\nabla_X,\nabla_Y]=\nabla_{[X,Y]}+\l(X,Y)\cdot id
\]
where $\l$ is a certain cocycle, $id$ is the identity operator.
\end{theorem}
In \cite{WZWN2}, \refT{pr_flat} has been formulated and proven in
the Conformal Field Theory setup, i.e.for a certain moduli space
of Riemann surfaces with marked points and fixed jets of local
coordinates at those points. We assert that the situation here is
quite similar and the proof is the same. In analogy with CFT we
refer to the projective flat connection defined by \refT{pr_flat}
as {\it Knizhnik-Zamolodchikov} connection.

The horizontal sections of the sheaf of covariants are called {\it
conformal blocks}.

\subsection{Representation of Hamiltonian vector fields
and commuting Hamiltonians}\label{S:rep}

By \refT{pr_flat} $X\to\nabla_X$ is a projective representation of
the Lie algebra of vector fields on $\P^D$ in the space of
sections of the sheaf of covariants. Denote this representation by
$\nabla$. The restriction of $\nabla$ to the subalgebra of
Hamiltonian vector fields gives the projective representation of
that.
\begin{theorem}
If $X$, $Y$ are Hamiltonian vector fields such that their
Hamiltonians Poisson commute then
$[\nabla_X,\nabla_Y]=\l(X,Y)\cdot id$. If the Hamiltonians depend
only on action variables, then $[\nabla_X,\nabla_Y]=0$.
\end{theorem}
\begin{proof} The projective commutativity immediately follows from
\refT{pr_flat} since $[X,Y]=0$.

For a Lax equation the spectral curve is an integral of motion.
This means that the complex structure, hence the transition
functions, are invariant along the phase trajectories. Hence
$\rho(X)=0$ and $\nabla_X=\partial_X$. This implies the
commutativity of Hamiltonians depending only on the action
variables.
\end{proof}
\subsection{Unitarity}\label{S:Unit}
The goal of the section is to specify a subspace of the
representation space and introduce a scalar product there such
that the above representation $\nabla$ becomes unitary.

Recall that in general only real vector field Lie algebras admit
unitary representations in the classical sense, i.e. such that the
representation operators are skew-Hermitian. For
complexifications, a convenient way to define unitarity is as
follows \cite{KaRa}. Let $\mathcal G$ be a Lie algebra and $T$ its
representation in the space $V$. Consider an antilinear
antiinvolution $\dag$ on $\mathcal G$ ($X\to X^\dag$,
$X\in\mathcal G$). An Hermitian scalar product in $V$ is called
contravariant if $T(X)^\dag=T(X^\dag)$ where the $\dag$ on the
left hand side means the Hermitian conjugation. A pair consisting
of $T$ and a contravariant scalar product is called a unitary
representation of $\mathcal G$. The restriction of $T$ to the Lie
subalgebra of the elements such that $X^\dag=-X$ is unitary in the
classical sense.

To construct a contravariant Hermitian scalar product in the space
of the representation $\nabla$ we first introduce a point-wise
scalar product in the sheaf of covariants, and then integrate it
over the phase space $\P^D$ by an invariant volume form.

Over every point in $\P^D$ we introduce an Hermitian scalar
product in ${\mathcal F}^{\infty/2}$ declaring semi-infinite
monom\-ials with basis entries to be orthonormal
(\cite[p.39]{KaRa}, \cite{KNFb}). Below (\refT{unit}) we prove
that it gives a well-defined point-wise scalar product on
coinvariants.

By Poincar\'e theorem the symplectic form and its degrees are
absolute integral invariants of Hamiltonian phase flows. Hence
$\w^p/p!$ defines an invariant measure with respect to Hamiltonian
phase flows. Let $\L^2(C,\w^p/p!)$ be the space of quadratically
integrable sections of the  sheaf $C$ with respect to that measure
where the square over a point is given by the above point-wise
scalar product.
\begin{theorem}\label{T:unit}
The representation $\nabla : X\to\nabla_X$  of the Lie algebra of
Hamiltonian vector fields on $\P^D$ in the subspace of smooth
sections  in $\L^2(C,\w^p/p!)$ is unitary.
\end{theorem}
\begin{proof}
First let us prove that the point-wise scalar product on
${\mathcal F}^{\infty/2}$ is well-defined on coinvariants. The
different quasihomogeneous almost-graded components of the module
${\mathcal F}^{\infty/2}$ are orthogonal. The point-wise
coinvariants are defined as $\V/\g^{\rm reg}\V$. The degrees of
mon\-omials occuring in the subspace $\g^{\rm reg}\V$ (including
summands in linear combinations) are obviously less then those of
monom\-ials which form coinvariants. Hence these spaces of
mono\-mials are orthogonal and the induced scalar product on the
quotient does not depend on the choice of representatives in the
equivalence classes, and we get a well-defined point-wise scalar
product on coinvariants.

Next let us construct a certain antilinear antiinvolution on the
tangent vector fields on $\P^D$. We will push it down from $\V$ by
the inverse to the Kodaira-Spencer mapping making use of the
double coset description of the tangent space to moduli space of
curves ($T_L {\mathcal M}=\V_-^{(1)}\backslash\V/\V_+^{(1)}$). The
Lie subalgebras $\V_\pm^{(1)}$ will be defined in the end of the
paragraph. Recall from \cite[p.39]{KaRa} that the convenient
antiinvolution in $\V$ is induced by its embedding into the Lie
algebra ${\mathfrak a}_\infty$ of infinite matrices with finite
number of diagonals. Following \cite{KaRa} we denote this
embedding by $r$. The antilinear antiinvolution in ${\mathfrak
a}_\infty$ amounts in $r(e_i)\to -r(e_i)^t$ for the elements of
the Krichever-Novikov base, with the consequent antilinear
continuation to the complexification. We go on denoting this
antiinvolution by $\dag$. For example in the case of two marked
points $\{P_1,P_\infty\}=\{P_+,P_-\}$ the $\dag$ sends the
subspace $\V^{(p)}_+$ of vector fields regular at $P_+$ and
vanishing there with the order at least $p$ to the similar
subspace $\V^{(p)}_-$ at $P_-$ (for an arbitrary $p\in\Z$). Hence
$\V^{(p)}_+\oplus\V^{(p)}_-$ is invariant under the
antiinvolution. By \refP{kernel} $\V^{(1)}_+\oplus\V^{(1)}_-=\ker
\rho^{-1}$ in this case. Hence $\dag$ is well defined on the
tangent space to $\P^D$ at the corresponding point.

For a local vector field $X$ on $\P^D$ by $\partial_X$ we mean the
corresponding derivative of local (for example finitary) smooth
sections of the sheaf~$C$. For a Hamiltonian $X$ the $\partial_X$
is skew-Hermitian by $X$-invariance of the measure, hence
$\partial_X^\dag=-\partial_X$. Further on, for a real Hamiltonian
vector field $X$ (i.e. $X^\dag=-X$) we have
$-\partial_X=\partial_{X^\dag}$, hence
$\partial_X^\dag~=~\partial_{X^\dag}$. By complex antilinearity we
obtain the same relation for all Hamiltonian vector fields.

Let $\langle\cdot |\cdot\rangle$ be the above introduced
point-wise scalar product, and $(\cdot |\cdot
)=\frac{1}{p!}\int_{\P^D}\langle\cdot |\cdot\rangle\w^p$. By
\cite{KNFb}  $\langle\cdot |\cdot\rangle$ is a contravariant form
with respect to the Sugawara representation (in the abelian case).
This implies that the $(\cdot |\cdot )$ possesses the same
property, hence $T(\rho(X))^\dag=T(\rho(X)^\dag)$. By definition
$\rho(X)^\dag=\rho(X^\dag)$. Hence
$T(\rho(X))^\dag=T(\rho(X^\dag))$.

By the last two paragraphs we have for a Hamiltonian vector
field~$X$
\[ \partial_X^\dag=\partial_{X^\dag}\
\text{and}\ {T(\rho(X))}^\dag=T(\rho(X^\dag)).
\]
Hence
\[
 (\partial_X+T(\rho(X)))^\dag=\partial_{X^\dag}+T(\rho(X^\dag))),
\]
i.e.
\[
 (\nabla_X)^\dag=\nabla_{X^\dag}.
\]
\end{proof}


}

\end{document}